\documentclass[final]{siamltex}
\usepackage{amssymb, amsmath}
\usepackage{float,epsfig}
\usepackage{algpseudocode}
\usepackage{algorithm}
\usepackage{enumitem}
\usepackage{tikz}
\usepackage{color}
\usepackage{multirow}
\usepackage{cleveref}
\usepackage{braket}
\usepackage{subfigure}
\usepackage[utf8]{inputenc}
\usetikzlibrary{patterns}
\usetikzlibrary{graphs}
\usepackage{algpseudocode}
\usepackage{algorithm}
\usepackage{comment}
\usepackage{tikz-network}
\usepackage{caption}
\usepackage{subcaption}
\usetikzlibrary{graphs,graphs.standard,calc}

\usepackage{tikz}
\newcommand{\Arrow}[1]{%
	\parbox{#1}{\tikz{\draw[->](0,0)--(#1,0);}}
}

\newcommand{\rsss}{\rotatebox[]{90}{$\boxminus$}\kern-0.7em{\mathrel{\raisebox{.2ex}{\Arrow{.35cm}}}}\!\!}

\newcommand{\csss}{\text{$\boxminus\kern-0.655em{\mathrel{\raisebox{-.2ex}{\rotatebox[]{-90}{\Arrow{.34cm}}}}}$\ }}

\newtheorem{exam}[theorem]{\bf Example}
\newtheorem{observation}[theorem]{ Observation}
\newcommand{\ba}{\begin{array}}
\newcommand{\ea}{\end{array}}

\newtheorem{note}{Note}
\newcommand{\be}{\begin{equation}}
\newcommand{\ee}{\end{equation}}
\newcommand{\beano}{\begin{eqnarray*}}
\newcommand{\eeano}{\end{eqnarray*}}

\def \Z{{\mathbb Z}}
\def \A{{\mathcal A}}

\def \d{{\bf d}}
\def \r{{\bf r}}

\def \diag{\mathrm{diag}}

\def \Im{\mathsf{Im}}

\title{Arithmetical Structures on Fan Graphs}
\author{Dilli Ram Chhetri \thanks{Department of Mathematics, Sikkim University, Sikkim-737102, India, ({\tt drchhetri.22pdmt01@sikkimuniversity.ac.in})} \and  Namita Behera \thanks{Department of Mathematics, Sikkim University, Sikkim-737102, India, ({\tt nbehera@cus.ac.in}, niku.namita@gmail.com)}  
\and Raj Bhawan Yadav \thanks{Department of Mathematics, Sikkim University, Sikkim-737102, India ({\tt rbyadav01@cus.ac.in})}}

\begin{document}

\maketitle

\begin{abstract}
In this paper, we study the arithmetical structures on Fan Graphs $F_n$. Let $G$ be a finite and connected graph. An arithmetical structure on $G$ is a pair $({\bf d}, {\bf r})$ of positive integer vectors such that $\textbf{r}$ is primitive (the greatest common divisor of its coefficients is 1) and $(\diag(\textbf{d}) - A)\textbf{r} = 0, $ where $ A$ represents the adjacency matrix of $G$.
This work explores the combinatorial properties of the arithmetical structures associated with $F_n. $ Further, we discuss the arrow-star graph, a structure derived from the fan graph, along with its properties.  Additionally, we investigate the critical group linked to each such structure on $F_n. $
\end{abstract}



\section{Introduction}
In recent years there has been a great deal of interest in studying the arithmetical structures on graphs. In this paper, we focus on investigating the arithmetical structures on Fan graphs.

Let $G= (V,E)$ be a finite and connected graph with $V$ is the vertex set and $E\subset V\times V$ represents the edge set. An arithmetical structure on $G$ is defined as a pair $({\bf d}, {\bf r})$, where ${\bf d}$ and ${\bf r}$ are positive integer vectors, ${\bf r}$ is primitive (the greatest common divisor of its coefficients is 1) and $$(\diag({\bf d}) - A){\bf r} = 0,$$ with $A$ representing the adjacency matrix of $G,$ where the adjacency matrix $A=[a_{ij}]$ associated with a graph is a symmetric matrix with $a_{ij}=1$ if $(i,j)\in E$, and $a_{ij}=0$ otherwise, for $i,j\in V$.

This definition extends the concept of the Laplacian arithmetical structure,  where ${\bf d}$ represents the vector of vertex degrees and ${\bf r} = {\bf 1} = (1, 1, \ldots, 1).$ Note that ${\bf d}$ and ${\bf r}$ uniquely determine each other, allowing us to regard, $\textbf{d}, \textbf{r},$ or the pair $({\bf d}, {\bf r})$ as an arithmetical structure on $G$. To avoid ambiguity, we will sometimes refer to these as arithmetical  $d$-structures or arithmetical $r$-structures  \cite{BHDNJC18}. The set of all arithmetical structures on $G$ is denoted Arith$(G)$ and sometimes by $\mathcal{A}(G)$. The data ($G, {\bf d}, {\bf r}$) together define an arithmetical graph. We denote the set of $\r$-structures of $G$ as $\A_{\r}(G).$

Arithmetical structures are deeply intertwined with number theory, algebraic geometry, and potential theory on graphs, serving as generalizations of concepts such as divisor theory on algebraic curves and abelian sandpile models \cite{biggs1997algebraic, HCEV18}. Their study intersects with areas like chip-firing games and the analysis of critical groups (or Jacobians) of graphs, which find applications in diverse fields, including network theory and statistical physics \cite{baker2007riemann,Corry2018DivisorsAS}.

Arithmetical structures on a graph, introduced by Lorenzini in \cite{LD89}, originate from intersection matrices in the study of degenerating curves in algebraic geometry. Here, the vertices of $G$ represent curve components, edges signify their intersections, and the entries of $\textbf{d}$ denote self-intersection numbers. These structures generalize the Laplacian matrix, encapsulating key graph properties.
Key concepts in this area that have garnered significant attention include the sandpile group and the chip-firing game (see \cite{CP18} for details). The critical group corresponds to the group of components of the Néron model of the Jacobian of the generic curve, an observation attributed by Lorenzini to Raynaud \cite{LD89, LP10}.

As in the classical case, the matrix $L(G, {\bf d}) = \diag({\bf d}) - A,$ where $A$ is the adjacency matrix of $G$, has rank $n - 1$, [\cite{LD89}, Proposition 1.1] with nullspace spanned by the all-ones vector $\bf{1}.$  The torsion  subgroup of the cokernel of $L$ forms the critical group of the arithmetical graph. Specifically, the critical group of an arithmetical structure corresponds to the torsion part of the cokernel of $L(G, {\bf d}) = \diag({\bf d}) - A,$ where $\diag({\bf d})$ is the diagonal matrix with the vector $\textbf{d}$ on the diagonal. In this setting, ${\bf r}$ generates the kernel of $L(G, {\bf d}).$ Further,  if we regard $L$ as a $\mathbb{Z}$-linear transformation $\mathbb{Z}^{n} \rightarrow \mathbb{Z}^{n}$, the cokernel $\mathbb{Z}^n / im L$ has the form $\mathbb{Z} \oplus \Phi(G)$; where $\Phi(G)$, the critical group is finite abelian, with cardinality equal to the number of spanning trees of $G$, by the Matrix-Tree Theorem. The critical group is also known as the sandpile group or the Jacobian. The elements of the critical group represent long-term behaviours of the well-studied abelian sandpile model on $G$; see, \cite{selforganizedcriticality1988, LP10,lionelleivnewhatisasandpile}.

The concept of a critical group first appeared in the literature in 1970 in the context of arithmetic geometry \cite{raynaud1970specialisation}. Later, Bak, Tang, and Wiesenfeld introduced the idea of self-organized criticality around 36 years ago \cite{selforganizedcriticality1988}, a concept widely studied in statistical physics and other domains. Its applications are detailed in P. Bak's book \cite{applicationofSelfOrganizedCriticality}.

In subsequent developments, D. Dhar's introduction of the abelian sandpile model \cite{dhar1990self} marked a paradigm shift. Informally described as a cellular automaton, its cells (vertices of a rooted graph) hold grains of sand. Dhar identified recurrent configurations in the model, forming a finite abelian group—later termed the sandpile group—with an order equal to the graph's number of spanning trees. This model has been extensively studied (e.g., \cite{AbelianAvalancesandTutttePolynomial,SandpileCobinatorialGameandCellularAutomata}).

Related concepts, such as chip-firing games, were explored by other authors under different frameworks (e.g., \cite{Chipfiringgamesongraph1991,NBiggsChipFiring,NBiggsTuttePolynomial}).

The group $\Phi(G)$ appears in the literature under various names, depending on its context of use:
\begin{enumerate}
\item Group of Components: Introduced in 1989 during a graph-theoretical study motivated by arithmetic geometry, following the terminology from the field \cite{lorenzini1991finite}.
\item Picard Group: Studied in 1997 by Bacher et al., linked to algebraic curves and shown to be isomorphic to the Jacobian group \cite{bacher1997lattice}.
\item Critical Group: Examined by Biggs in 1999 in the context of chip-firing games \cite{biggs1997algebraic}.
\item Smith Group: In 1990, associated with the Smith normal form of the adjacency matrix $A$ of $G$ \cite{rushanan1990combinatorial}.
\end{enumerate}


Recently, the critical groups of arithmetical structures on star graphs and complete graphs have been studied in \cite{arithmeticaloncompletegraphs}. Additionally, \cite{criticalpolynomialofgraphbyLorenzini} examines the critical polynomial of a graph, which is defined as the determinant of a specific matrix associated with a connected graph $G$. Building on these studies, this article investigates how every graph can be endowed with an arithmetical structure whose associated critical group is trivial. In \cite{BHDNJC18}, critical groups of arithmetical structures on path and cycles have been discussed. Critical groups of arithmetical structures under a generalized star-clique operation has been calculated in \cite{AAJL18}. 
The sandpile group of a planar graph is isomorphic to that of its dual, as shown in \cite{SandpileGroupofDualofGraph}. The sandpile group of the cone of a graph is studied in \cite{SandpileGroupofConeofGraph}. Additionally, literature establishes that every surjective uniform homomorphism of graphs induces an injective homomorphism between their sandpile groups and explores relationships between the sandpile group of the cone of the Cartesian product of graphs and that of their factors. Connections between the sandpile group of a graph $G$ and its line graph are also discussed in \cite{sandpileGroupofLineGraphs}. The sandpile group has been fully determined for certain families of graphs, including those discussed in \cite{sandpileGroupofThreshholdGraph, sandpileGroupofTree, SandpilegroupoverEllpiticCurves, sandpileGroupofCompleteMultipartiteGraph, sandpileGroupofRegularTrees}. Additionally, the sandpile group of the Cartesian product of graphs and hypercubes has been extensively studied in \cite{SandpileGroupofCratesianProduct1, SandpileGroupofCratesianProduct2, SandpileGroupofCratesianProduct3, SandpileGroupofCratesianProduct4, SandpileGroupofCratesianProduct5}.


In \cite{HCEV18}, arithmetical structures on complete graphs, paths, and cycles are studied, starting with a general perspective using $M$-matrices \cite{Berman1994andPlemmons}. The work reaffirms Lorenzini's result on the finiteness of arithmetical structures on a graph and shows that their count on a path corresponds to the Catalan numbers \cite{catalannumberbycarlitz}. Similarly, for cycles, \cite{BHDNJC18} proves that arithmetical structures are enumerated by the binomial coefficients $\binom{2n-1}{n-1}$ , and obtain refined enumeration results related to multisets.



Apart from the arithmetical structures on the path and cycle, in general the description of the arithmetical structures on a
graph is a very difficult problem. Therefore in \cite{BHDNJC18} it is studied the arithmetical structures on a graph with a cut vertex. For instance, for the complete and star graph \cite{arithmeticaloncompletegraphs}, its arithmetical structures
are in one to one correspondence with a variant of the Egyptian fractions  \cite{ Egyptianfraction2, Egyptianfraction3}. For example, the arithmetical $d$-structures on the star $K_{n,1}$ can be shown to be the positive integer solutions to the equation $d_0 = \sum_{i=1}^{n}\frac{1}{d_i}.$


It is known that Arith$(G)$ is finite for all connected graphs $G$ [\cite{LD89}, Lemma 1.6], though the proof relies on a non-constructive approach using Dickson’s Lemma. The finiteness result of arithmetical structure on connected graph as given by Lorenzini was recovered by Hugo Corrales and Carlos E. Valencia in \cite{HCEV18} from the perspective of M-matrices \cite{Berman1994andPlemmons}. This raises the challenge of explicitly enumerating arithmetical structures for specific graphs or families of graphs. Recent works \cite{BHDNJC18, KAAL20, PathwithDoubleedgeArithstructure} have addressed this for various graph families, including path graphs, cycle graphs, bidents, and certain path graphs with doubled edges. \cite{keyesReiter2021boundingthenumberofArithstructure} has studied certain upper bound for the number of arithmetical structures on $G$ depending only on the number of vertices and edges of $G$. Additionally, arithmetical structures on wheel graphs have also been studied in recent research \cite{BANBDCRBY2024}. Motivated by these developments, we investigate the arithmetical structures on the Fan graph $F_n$.

Fan graph (or Dutch windmill graph or n-fan) $F_n$ is a planar, 
undirected graph with $2n + 1$ vertices and $3n$ edges.
The Fan graph $F_n$ can be constructed by joining n copies of the cycle graph $C_3$ with a common vertex, which becomes a universal vertex for the graph. See figure \ref{fig: Friendship Graphs}.
\begin{figure}[hbt!]
\centering
\includegraphics[width=.8\textwidth]{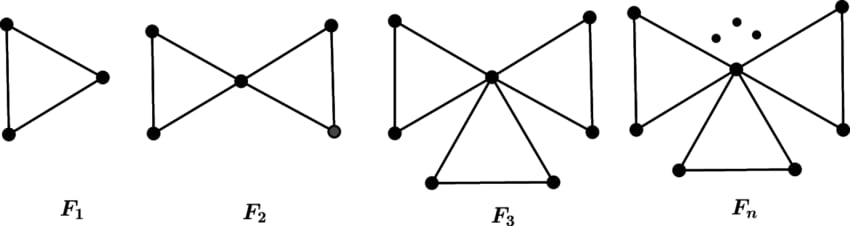}
\caption{ Fan Graphs}
\label{fig: Friendship Graphs}
\end{figure}

In this paper, we consider $F_n$ and study the arithmetical structures on Fan graph $F_n$.

The rest of the paper is organized as follows.
In section~2 we review essential definitions and results on matrices that are used throughout the paper. In section 3 an analysis of the arithmetical structures on the Fan graph $F_n$ is presented. Section 4 discusses the arrow-star graph, a structure derived from the fan graph, along with its properties. In section 5 the critical group of arithmetical structures on Fan graphs is described.

\section{Preliminaries}
In this section, we present some fundamental definitions and results that will be used throughout the paper.
\begin{definition}
A square matrix $A$ is called reducible if there exists a permutation matrix $P$ such that $$PAP^{t} = \begin{pmatrix}
    A_1 & * \\
    0 & A_2 \\
\end{pmatrix}$$ for some non-trivial square matrices $A_1$ and $A_2.$    
\end{definition}

\begin{definition} [Neighbourhood of a vertex v in a graph G] \cite{west2001introduction} 
 The neighbourhood of a vertex v in a graph G is the subgraph of G induced by all vertices adjacent to v, i.e., the graph composed of the vertices adjacent to v and all edges connecting vertices adjacent to v.  The neighbourhood is often denoted by $N_G(v)$.
\end{definition}

Here we recall the classical concept of an $M$-matrix and study a class of $M$-matrices whose proper principal minors are positive and its determinant is non-negative. Let us begin with some definitions: \cite{HCEV18}

\begin{definition}[Real non-negative matrix] \cite{HCEV18}
A real square matrix is called non-negative if all its entries are non-negative real numbers.    
\end{definition}

\begin{definition}[$Z$-matrix] \cite{HCEV18}
 A real matrix $ A = (a_{i,j}) \in \mathbb{R}^{n \times n}$ is called a Z-matrix if $a_{i,j} \leq 0 ,\ for \ all \ i \neq j$. 
\end{definition}

\begin{definition}[$M$-matrix] \cite{HCEV18}
 A $Z$-matrix $A$ is an $M$-matrix if there exists a non-negative matrix $N$ and a non-negative number $\alpha$ such that such that $A = \alpha I- N \ and \ \alpha \geq \rho(N)$, where $\rho(N) = max\{|\lambda|  : \lambda \in \sigma(N)\}$.   
\end{definition}
The study of $M$-matrices can be divided into two major parts: non-singular M-matrices and singular $M$-matrices (see[\cite{Berman1994andPlemmons}, Section 6.2 and 6.4])

\begin{definition}
 \cite{HCEV18} A real matrix $A =(a_{i,j})$ is called an almost non-singular $M$-matrix if $A$ is a $Z$-matrix, all its proper principal minors are positive and its determinant is non-negative.    
\end{definition}

\begin{theorem} \cite{HCEV18}\label{Theorem 2.6 from Corrales and Valencia}
 If $M$ is a real $Z$-matrix, then the following conditions are equivalent: 
 \begin{itemize}
     \item [(1)] $M$ is an almost non-singular $M$-matrix.
     \item [(2)] $M +D$ is a non-singular $M$matrix for any diagonal matrix $D>_{\neq} 0$. 
     \item [(3)] $\det(M) \geq 0$ and $\det(M +D) {\gneq} \det(M +D'){\gneq} 0$  for any diagonal matrices $D >_{\neq} D'{\gneq} 0.$  
 \end{itemize} 
\end{theorem}

\begin{theorem} \cite{HCEV18}\label{Theorem 3.2 from Corrales and Valencia}
Let $M$ be a $Z$-matrix. If there exists ${\bf r}$ with all its entries positive such that $M{\bf r}^{t} = 0^t,$ then $M$ is an $M$-matrix. Moreover, $M$ is an almost non-singular $M$-matrix with $\det(M) = 0$ if and only if $M$ is irreducible and there exists ${\bf r}$ with all its entries positive such that $M{\bf r}^t = 0^t. $   
\end{theorem}

\section{Arithmetical Structures on Fan Graphs}
In this section, we consider a fan graph $F_n$ and study the arithmetical structures on $F_n$.

We consider $F_n$ to be nontrivial when $n\geq 2.$ That is, $F_n$ is a fan graph with vertices   $v_0, v_1,  \ldots v_{2n} $,  where $v_0$ is the central vertex with degree $2n$ and $ v_1 \ldots, v_{2n}$ are the vertices of cyclic arms with degree $2$ each and $ v_{2n+1} = v_1$. Then the adjacency matrix $A(F_n)=[a_{v_i,v_j}]$ of order $2n+1$ is a symmetric matrix given by:
 \begin{equation}\label{eqn:adjmwn}
    a_{v_i,v_j}=\begin{cases}
        1, &  \mbox{if} \; i=0, j=1,\hdots, 2n \\
        1, & \mbox{if} \;  i\; is \; odd,\; j=i+1, 0 \\
        1,  & \mbox{if} \; i\; is \;  even,\; j=i-1, 0 \\
        0, & \,\, \mbox{otherwise}.
    \end{cases}.
\end{equation} 
Now, consider the adjacency matrix $A(F_n)$. Then we have the following result.
   
\begin{proposition}
Consider the Fan Graph $F_n$ in $2n+1$ vertices and $3n$
 edges. Let $\textbf{r}= (r_0, r_1, \ldots, r_{2n}) $ be a primitive positive integer vector. Then $\textbf{r}$  is an arithmetical $r$-structure on $F_n$ iff it satisfies the following conditions:
\begin{equation}\label{r-structures on F_n}
\begin{aligned}
 & r_0 | r_1 + r_2 + \cdots + r_{2n} \\
 & r_i | r_0 + r_{i+1}   \  if \  i \ is \ odd\\
 & r_i | r_0 + r_{i-1} 
 \  if \ i \ is \ even
\end{aligned}
\end{equation}
\end{proposition}

\begin{proof}
Note  that $(\diag(\textbf{d)} - A(F_n)){\bf r }= 0$ is equivalent to 
\begin{equation} 
  \begin{aligned}
 & d_0r_0 = r_1 + r_2 + \cdots + r_{2n} \\
 & d_ir_i = r_0 + r_{i+1}   \  if \  i \ is \ odd\\
 & d_ir_i = r_0 + r_{i-1} 
 \  if \ i \ is \ even
\end{aligned},
    \end{equation} which is equivalent to  the required divisibility conditions on $r$-structure given in (\ref{r-structures on F_n}). Hence, proved. 
\end{proof}

\begin{note}
If $r$-structure satisfies (\ref{r-structures on F_n}), then the corresponding $d$-structure can be derived from the definition of arithmetical structure on Fan graphs.
\end{note}

\begin{note}
The arithmetical structure on $F_n$ derived from the Laplacian of $F_n$ has ${\bf r}= (1, 1, \ldots, 1) $ and ${\bf d}= (2n, 2, \ldots, 2) $. We call this pair \textbf{(d,r)} the Laplacian arithmetical structure.
\end{note}

\begin{exam}
Consider graph $F_2.$ Then the following are some $r$-structures on $F_n$.
$$(1, 1, 1, 1, 1), (1, 1, 2, 1, 2), (1, 2, 3, 2, 3), (1, 1, 2, 1, 1), (1, 2, 3, 2, 1).$$
\end{exam}

\begin{note}
Alternatively, $\textbf{(d,r)}$ is an arithmetical structure on Fan graph $F_n$ iff label $r_i$ at vertex $v_i$ divides the sum of the labels at adjacent vertices of $v_i$, i.e., the neighbouring vertices of $v_i$. 
\end{note}

Before moving to next result we will fix some notations. If $\textbf{d} , \textbf{a} \in $ $\mathbb{R}^V $ then, we say that $\textbf{d} \leq \textbf{a}$ if and only if $\textbf{d}_v \leq \textbf{a}_v$ $\forall \, v \in V$, where $\leq$ is a partial order in $\mathbb{R}^V$.
For $\textbf{d} , \textbf{a} \in $ $\mathbb{R}^V $  we say that $\textbf{d} < \textbf{a}$ if and only if $\textbf{d} \leq \textbf{a}$ and $\textbf{d} \neq \textbf{a}$. Now, we have the following theorem:

\begin{theorem} \label{Theorem for unique Laplacian Arith Structure on fan graph}
Let $F_{n}$ be the Fan graph and $\textbf{(d,r)}\in {\bf A}(F_n).$  If ${\bf r} \neq {\bf 1}_{2n+1}^t$ then one of the following holds:
 \begin{enumerate}
    \item $ d_0> 2n$ and $d_i<2,$ for some $i>0$.
    \item $d_0<2n$ and $d_i > 2$, for some $i>0$.
    \item $d_0=2n$ and $d_i<2$, $d_j>2,$ for some $0< i,j\le 2n.$
\end{enumerate}
If $d_i=1$ for some $i$,  then $\textbf{d} _u >1$  for all $u \in N_{F_{n}}(v_i)$, where $N_{F_{n}}(v_i)$ is the neighbourhood of $i$th vertex $v_i$ in $F_{n}$.
\end{theorem}

\begin{proof} 
Let $F_{n}$ be a Fan graph with vertices   $v_0, v_1, v_2,  \ldots v_{2n} $,   where $v_0$ is the central vertex with degree $2n$ and $v_1, v_2 \ldots, v_n$ are the vertices of cycle with degree $2$ each. Since $\textbf{(d,r)} \neq ((2n, 2{\bf 1}_{2n}^t), {\bf 1}_{2n+1}^t) \in Arith (F_{n})$ and $0 < \textbf{d}$, by using Theorem (\ref{Theorem 2.6 from Corrales and Valencia}) and Theorem (\ref{Theorem 3.2 from Corrales and Valencia}), we get that neither $ ((2n),2,2,2, \cdots,2) \geq \textbf{d}$ nor  $ \textbf{d}\geq((2n),2,2,2, \cdots,2)$. From this we conclude that  one of the following holds: 
  \begin{enumerate}
    \item $ d_0> 2n$ and $d_i<2,$ for some $i>0$.
    \item $d_0<2n$ and $d_i > 2$, for some $i>0$.
    \item $d_0=2n$ and $d_i<2$, $d_j>2,$ for some $0< i,j\le 2n.$
\end{enumerate}
  Now as $L(F_{n},\textbf{d}) \textbf{r}^t = \textbf{0}^t$, we get 
  \begin{equation} \label{nmt1}
  \begin{aligned}
 & d_0r_0 = r_1 + r_2 + \cdots + r_{2n} \\
 & d_ir_i = r_0 + r_{i+1}   \  if \  i \ is \ odd\\
 & d_ir_i = r_0 + r_{i-1} 
 \  if \ i \ is \ even
\end{aligned}
    \end{equation}
Let $d_i=1$ for some $i\ge 0$. Consider equation $(\ref{nmt1})$.\\

{\bf Case I: } Suppose that $i=0$ and $d_j=1$, for some odd $j$. Then using (\ref{nmt1}) we have $r_0 =r_1+r_2+\cdots+r_{2n}$ and $r_j=r_0+r_{j+1}$. This implies that   $r_0=r_1+r_2+\cdots+r_{j-1}+r_0+ r_{j+1}+\cdots+ r_{2n}$.  It implies that $r_1+r_2+\cdots+ r_{j-1}+ 2r_{j+1}+\cdots+ r_{2n}=0$. Which is a contradiction to the fact that all the $r_i's$ are positive integers.  Similarly, 
suppose that $i=0$ and $d_j=1$, for some even $j$. Then using (\ref{nmt1}) we have $r_0 =r_1+r_2+\cdots+r_{2n}$ and $r_j=r_0+r_{j-1}$. This implies that   $r_0=r_1+r_2+\cdots+r_{j-1}+r_0+ r_{j-1}+ r_{j+1}+\cdots+ r_{2n}$.  It implies that $r_1+r_2+\cdots+ 2r_{j-1}+ r_{j+1}+\cdots+ r_{2n}=0$. Which is a contradiction to the fact that all the $r_i's$ are positive integers. \\

{\bf Case II: } Suppose that $i$ is even, $i>0$ and $d_{i-1}=1$. Then by (\ref{nmt1}) we get $r_{i} = r_0+r_{i-1}$. Again, $r_{i-1} = r_0+r_{i}$. Therefore, we have $r_i = 2r_0 + r_i$. It implies that $r_0 = 0.$ which is also a contradiction.\\

{\bf Case III: } Suppose that $i$ is odd and $d_{i+1}=1$. Then by (\ref{nmt1}) we get $r_{i} = r_0+r_{i+1}$. Again, $r_{i+1} = r_0+r_{i}$. Therefore, we have $r_i = 2r_0 + r_i$. It implies that $r_0 = 0.$ which is also a contradiction.  
\end{proof}

\begin{corollary}\label{AS1}
There is exactly one arithmetical structure $(\textbf{d}, \textbf{r})$ on $F_n$ such that $d_0\ge 2n$, $d_i \geq 2$ for all $i>1$, namely $$\textbf{d} = (2n, 2, 2, \ldots, 2) \text{     and    }  \textbf{r} = (1, 1, 1, \ldots, 1). $$ 
\end{corollary}
\begin{proof}
The proof follows as the consequence of Theorem \ref{Theorem for unique Laplacian Arith Structure on fan graph}.
\end{proof}

The following propositions provide a class of elements of $\A(F_n).$

\begin{proposition}
There is an arithmetical structure $(\textbf{d}, \textbf{r})$ on $F_n$ such that $\textbf{r} = (2n, 1, 1, \ldots, 1). $.    
\end{proposition}

\begin{proof}
Since $\textbf{r}$ satisfies (\ref{r-structures on F_n}), we have thus produced an arithmetical structure on $F_n$.    
\end{proof}

Let $S_{2n+1}$ be the group of permutations on $2n+1$ symbols $\{0,1 \ldots, 2n\}.$
Let $B_1=\{(l,m)(l+1,m+1)\in S_{2n+1}: l, m \;\text{are  odd} \} $ and $B_2=\{(l,l+1)\in S_{2n+1}:   l \;\text{is  odd }\}$. Let $H$ be the subgroup of $S_{2n+1}$ generated by $B_1\cup B_2.$  Denote the set of all bijective functions from $\mathcal{A}_r(F_n)$ to itself by $Iso(\mathcal{A}_r(F_n))$.
Define  $\rho: H\to Iso(\mathcal{A}_r(F_n))$ by 
$$\rho( \sigma)(r_0, r_1, \cdots, r_{2n})=(r_0, r_{\sigma(1)}, \cdots, r_{\sigma(2n)}).$$
Then $\rho$ defines group homomorphism from $H$ to $Iso(\mathcal{A}_r(F_n))$. Hence this defines an action of $H$ on $\mathcal{A}_r(F_n)$.


Now, we describe the $\rho$ actions of cyclic group $\mathbb{Z}_n$ on the set of arithmetical r-structure $\mathcal{A}_r(F_n)$ of $F_n$. For $c \in \  \mathbb{Z}_n$ we define 
 $$ \rho _c(r_0, r_1, \cdots, r_{2n}) = (r_0, r_{2c+1},\cdots,r_{2n},r_1,\cdots,r_{2c}) =(r_0,r_{2c\oplus 1},r_{2c\oplus 2}, \hdots, r_{2c\oplus 2n}),$$ where $\oplus$ is modulo $n$ addition for any $c\in \mathbb{Z}_n$  and $ (r_0, r_1, \cdots, r_{2n}) \in \mathcal{A}_r(F_n)$.  Denote the set of all bijective functions from $\mathcal{A}_r(F_n)$ to itself by $Iso(\mathcal{A}_r(F_n))$. Then $\rho_c$ defines group homomorphism from $\mathbb{Z}_n$ to $Iso(\mathcal{A}_r(F_n))$.

For example, $\rho$-orbit of arithmetical r-structure (1, 2, 3, 3, 2) on $F_2$ is given by: $(1, 2, 3, 3,2),(1, 3, 2, 2,3).$




The following theorem outlines how  arithmetical structures on 
$F_{n+1}$ are derived from those on $F_{n}$ under certain conditions.

\begin{theorem}
 Let $\textbf{(d,r})$ be an arithmetical  structure on $F_n$ such that $r_0$ divides $r_{2m-1} + r_{2m}$, for some $1\leq m\leq n$. Let $(\tilde{\textbf{d}},\tilde{\textbf{r}})$ be given by 
$$
\tilde{\textbf{d}} = \begin{pmatrix}
    d_0 + \frac{r_{2m-1} + r_{2m}}{r_0} \\ d_1 \\ \vdots \\ d_{2n} \\ d _{2m-1} \\ d_{2m}
\end{pmatrix} \,\,\,\,\,
\text{  and   } \,\,\, 
\tilde{\textbf{r}} = \begin{pmatrix}
    r_0\\ r_1\\ \vdots \\ r_{2n}\\ r_{2m-1}\\ r_{2m}
\end{pmatrix}.
$$ Then $ (\tilde{\textbf{d}},\tilde{\textbf{r}}) \in Arith (F_{n+1}).$
\end{theorem}

\begin{proof} Let $A(F_n)$ be the adjacency matrix of a fan graph $F_n$.
    Observe that  $$ \tilde{A} = A(F_{n+1}) = \begin{pmatrix}
        A(F_n) & P_{12} \\
        P_{21} & P_{22}
    \end{pmatrix},$$
    where $$P_{21} = \begin{pmatrix}
        1 & 0 & \cdots & 0\\
        1 & 0 & \cdots & 0
    \end{pmatrix} _{2 \times n}, \,\,\,\, P_{12} = P_{21}^t \,\,\,
    \text{         and      }
     P_{22} = \begin{pmatrix}
         0 & 1\\ 
         1 & 0
    \end{pmatrix}.$$
    Then we have $$(\diag (\tilde{\textbf{d}}) - \tilde{A}) \tilde{\textbf {r}} = 0.$$ Hence proved.
\end{proof}

\begin{corollary}
 Let $\textbf{(d,r})$ be an arithmetical  structure on $F_n$ such that $r_0$ divides $r_{2n-1} + r_{2n}$. Let $(\tilde{\textbf{d}},\tilde{\textbf{r}})$ be given by 
$$
\tilde{\textbf{d}} = \begin{pmatrix}
    d_0 + \frac{r_{2n-1} + r_{2n}}{r_0} \\ d_1 \\ \vdots \\ d_{2n} \\ d _{2n-1} \\ d_{2n}
\end{pmatrix} \,\,\,\,\,
\text{  and   } \,\,\, 
\tilde{\textbf{r}} = \begin{pmatrix}
    r_0\\ r_1\\ \vdots \\ r_{2n}\\ r_{2n-1}\\ r_{2n}
\end{pmatrix}.
$$ Then $ (\tilde{\textbf{d}},\tilde{\textbf{r}}) \in Arith (F_{n+1}).$
\end{corollary}

We observe that the Fan graph $F_n$ consists of $n$ cycle graphs $C_3$, each referred to as a cyclic arm. Specifically, $F_2$ has $2$ cyclic arms, $F_3$ has $3$ cyclic arms, and in general, $F_n$ has $n $ cyclic arms. Let $A_i$ denote the $i$-th cyclic arm. For convenience, we arrange the cyclic arms in a clockwise order as $A_1$, $A_2$, $\cdots$, $A_n$. Notably, $F_1$ consists of a single cyclic arm $A_1$, while $F_2$ has arms 
$A_1$, $A_2$, $F_3$ has $A_1$, $A_2$, $A_3$, and so on. In $F_n$, all arms meet at a single vertex, which we refer to as the central vertex.

We fix the vertex ordering of the Fan graph $F_n$ as $v_0$, $v_1$, $v_2$, $\cdots$, $v_{2n}$, in clockwise order. Here, $v_0$ is the central vertex, which has a degree of $2n$, while all other vertices belong to the cyclic arms and have  degree  $2$. The central vertex $v_0$ serves as the meeting point for all cyclic arms.

We label the cyclic arms as follows:
$A_1 = (v_0, v_1,v_2)$, $A_2 = (v_0, v_3,v_4)$, $A_3 = (v_0, v_5,v_6)$, $\cdots$, $A_r = (v_0, v_{2r-1},v_{2r})$, $\cdots$, $A_{2n} = (v_0, v_{2n-1},v_{2n})$.
Finally, we observe a strong relationship between the arithmetical structure on each cyclic arm $A_i$ and the overall arithmetical structure on the fan graph $F_n$.

Let us revisit the recurrence relation of $F_n$. For any 
$(\textbf{d},\textbf{r}) \in Arith(F_n)$, we observe that, based on Equation (\ref{r-structures on F_n}), the recurrence relation of $F_n$ follows a specific pattern: \\

(a) In the first equation, the $r$-label at the central vertex $v_0$ divides the sum of the $r$-labels at its adjacent vertices.

This property reflects the centrality of $v_0$, as it connects all the cyclic arms $A_1$, $A_2$, $\cdots$, $A_n$ and determines the arithmetic relationships within the fan graph $F_n$. Each adjacent vertex of 
$v_0$ belongs to one of the cyclic arms, contributing to its arithmetical structure. \\

(b) The remaining recurrence relations (except the first one) in $F_n$ are identical to the recurrence relations for the cycle graph $C_3$, with the exception that $r_0 | r_i + r_{i-1} \ for \ 1 < i \leq 2n$ i.e., $r_0 | r_i + r_{i-1} \ for \  1 < i \leq 2n$ need not hold in general for $F_n.$

Let $(\textbf{d}_i,\textbf{r}_i) \in Arith ( C_3)$, $1\leq i\leq n$ be $n$ arithmetical structures of $C_3$  where 
$\textbf{d}_i = (d_{i1}, d_{i2}, d_{i3})$ and $\textbf{r}_i = (r_{i1}, r_{i2}, r_{i3})$, with $r_{ij}$  being the  the $j^{th}$ entry of $\textbf{r}_i$. Let 
 $r_{i1} = r_0 \ \forall 1\leq i\leq n.$.  To construct $\tilde{\textbf{r}}$  and $\tilde{\textbf{d}}$, proceed as follows:
 \begin{equation} \label{DNR}
\begin{aligned}
\tilde{\textbf{r}} &= ( r_0, r_{12}, r_{13}, \cdots, r_{i2}, r_{i3}, \cdots, r_{n2}, r_{n3})\\ and\\
\tilde{\textbf{d}} &= (\tilde{d}_0, \tilde{d}_1, \tilde{d}_2, \tilde{d}_, \cdots, \tilde{d}_i, \cdots, \tilde{d}_{2n}), \ where\\
\tilde{d}_0 &= \frac{ r_{12} + r_{13} +  \cdots + r_{n2} + r_{n3}}{r_0} \\
\tilde{d}_i& = \begin{cases}
    \frac{r_0 +r_{i3}}{r_{j2}} \,\,\,\ if \ i=2j-1.\\
    \frac{r_0 +r_{j2}}{r_{j3}} \,\,\, \ if \ i =2j.
\end{cases}
\end{aligned}
\end{equation}

With the above observations, we have the following theorem.

\begin{theorem} \label{Arith st. of F_n from C_3}
Let $(\textbf{d}_i,\textbf{r}_i) \in Arith ( C_3)$, $1\leq i\leq n$ be $n$ arithmetical structures of $C_3$, where 
$\textbf{d}_i = (d_{i1}, d_{i2}, d_{i3})$ and $\textbf{r}_i = (r_{i1}, r_{i2}, r_{i3})$, with $r_{ij}$  being the $j^{th}$ entry of $\textbf{r}_i$. Let 
 $r_{i1} = r_0 \ \forall 1\leq i\leq n.$ Then $(\tilde{\textbf{d}},\tilde{\textbf{r}})$ given by equation \ref{DNR} is an arithmetical structure on fan graph $F_n$.   
\end{theorem}

\begin{proof}
By direct computation one can verify the proof.
\end{proof}


\begin{exam}
Consider $F_2$. First label the arms of $F_2 $ as $A_1 = (r_0, r_{12},r_{13})$  and $A_2 = (r_0, r_{22},r_{23})$. Now, since we know $C_3 $ has 10 arithmetical structures see, [\cite{BHDNJC18}, Theorem 30], Then $$ \mathcal{C}_r = Arith( r \text{-structure  of  }  C_3) $$ $=  \{ (1,1,1), (1,2,3), (2,3,1), (3,1,2), (1,3,2), (3,2,1), (2,1,3), (1,1,2), (1,2,1), (2,1,1)\}.$ Now, we will construct r-structure and d-structure on $F_2$ from r-structure and d-structure of its arms. 
\begin{enumerate}
\item Take $\mathcal {C}
_r(A_1) = (1,1,1) = \mathcal{C}_r(A_2)$ in this case we will have $\tilde{\textbf{r}} = (1,1,1,1,1,1)$ usual r-struture of $F_2$.
\item Take  $\mathcal{C}_r(A_1) = (1,2,3) = \mathcal{C}_r(A_2)$, then we will have $\tilde{\textbf{r}} = (1,2,3,3,2)$. Now keeping first entry fixed rotate the entries of $\tilde{\textbf{r}}$  in $\rho$ orbit as described  previously. So that we will get  $\tilde{\textbf{r}} = (1,3,2,2,3)$.
\item Take  $\mathcal{C}_r(A_1) = (1,2,3), \mathcal{C}_r(A_2) = (1,3,2)$, then we have $\tilde{\textbf{r}} = (1,2,3,3,2)$.
\item Take  $\mathcal{C}_r(A_1) = (1,2,3), \mathcal{C}_r(A_2) = (1,1,1)$, then we have $\tilde{\textbf{r}} = (1,2,3,1,1)$. Now keeping first entry fixed rotate the entries of $\tilde{\textbf{r}}$  in $\rho$ orbit as described  previously. So that we will get $\tilde{\textbf{r}} = (1,1,1,2,3)$.
\end{enumerate}
In this way we can find $\tilde{\textbf{r}}  \ \in \ \mathcal{A}_r(F_2) = Arith( \ r-structure \ of \ F_2)$
\end{exam}

With the help of theorem \ref{Arith st. of F_n from C_3} we can give rough lower bound to the number of arithmetical structure on $F_n$.
\begin{theorem}
We have $| Arith ( F_n)| \geq 10^n$. 

\end{theorem}



We will introduce a notation here, i.e., $\textbf{r}(1)$ = $ \#\{i \ |  \textbf{r}_i = 1\}.$
\begin{lemma}
Let $(\tilde{\textbf{d}},\tilde{\textbf{r}})$ be an arithmetical r-structure of $F_n$ obtained from arithmetical r-structure of its arms $A_i$ then ${\tilde{\textbf{r}}(1) > 0}$ .
\end{lemma}
\begin{proof}
 since $\textbf{r}(1) > 0$ for r-structure on $C_3$. By our choice of $\tilde{\textbf{r}}$ we have  $\tilde{\textbf{r}}(1) > 0$.
\end{proof}

Given a fan graph $F_n$ containing $2n+1$ vertices $v_0$, $v_1$, $\cdots$, $v_{2n}.$ We define sub-division and smoothing operation as follows:

The sub-division operation adds a new cyclic arm say $A_k$ in between cyclic arms $A_i$ and $A_{i+1}$ with the r-labelling as $r_{k1}$, $r_{2n+1}$, $r_{2n+2}$ where  $r_{k1}$ is the r-labelling in the first vertex which is same as the r-labelling in the central vertex of $F_n$. 
The reverse operation smoothing will remove one cyclic arm say $A_k$ in between two cyclic arms $A_i$ and $A_{i+1}$ keeping the central vertex intact. We give following results relating the arithmetical structure on $F_{n+1}$ with the arithmetical structures on $F_n$ regarding the former graph as sub-division of the latter and vice-versa.

\begin{theorem} \label{subdivision and smoothing equation}
\begin{enumerate} 
\item Let $n \geq 3$ and  $(\textbf{d},\textbf{r}) \in Arith(F_n)$. Let $(r_{k1}, r_{2n+1}, r_{2n+2}) \in \mathbb{N}^3$ be such that  $r_{2n+1} \mid r_{k1} + r_{2n+2}$ , $r_{2n+2} \mid r_{k1} + r_{2n+1}$, $r_{k1} = r_{0}$ and $$r_{k1} = r_0  \mid \sum _{i= 1} ^ {2n+2 }r_i$$ where $\textbf{r} = (r_0, r_1, \cdots, r_{2n}) \in \mathcal{A}_{\textbf{r}}(F_n)$. Define integer vectors $\tilde{\textbf{d}}$ and $\tilde{\textbf{r}}$ of length $2n+3$ as follows:
\begin{equation}\label{Arithmetical r-structure of F_n to F_{n+1}}
\tilde{\textbf{r}_i} = \begin{cases}
  r_i \ for \ 0 \leq i \leq 2n,\\
   r_{2n+1} \ for \ i = 2n+1,\\
   r_{2n+2} \ for \ i = 2n+2,
\end{cases}
\end{equation}
\begin{equation} \label{Arithmetical d-structure of F_n to F_{n+1}}
\tilde{\textbf{d}_i} = \begin{cases}
  d_i \ for \ 0 \leq i \leq 2n,\\
  \frac {r_0 + r_{2n+2}}{r_{2n+1}} \ for \ i = 2n+1,\\
   \frac {r_0 + r_{2n+1}}{r_{2n+2}} \ for \ i = 2n+2.
\end{cases} 
\end{equation}
Then $(\tilde{\textbf{d}},\tilde{\textbf{r}})$ is an arithmetical structure on $F_{n+1}$.

\item Let $n \geq 3$ and $(\tilde{\textbf{d}},\tilde{\textbf{r}}) \in Arith(F_n)$ with $\tilde{r}_0 \mid \tilde{r}_i +\tilde{r}_{i+1}$ for $i$ odd. Define integer vectors $\textbf{d}$ and $\textbf{r}$ of length $2n-1$ as follows:
\begin{equation}\label{Arithmetical structure of F_n to F_{n-1}}
\begin{aligned}
\textbf{r}_i &= \begin{cases}
  \tilde{r}_i \ &for \ 0 \leq i \leq 2n-2\\
\end{cases} \\
\textbf{d}_i &= \begin{cases}
  \tilde{d}_i \ &for \ 0 \leq i \leq 2n-2,\\
  \end{cases} 
\end{aligned}
\end{equation}
where $k^{th}$ arm $ A_k = ( r_{k1}, r_{k2}, r_{k3})$ is removed keeping central vertex intact. Then $(\textbf{d},\textbf{r})$  is an arithmetical structure on $F_{n-1}$.
\end{enumerate}  
\end{theorem}

\begin{proof}
For the first part of Theorem \ref{subdivision and smoothing equation}, we need to prove that $ (\tilde{\textbf{d}},\tilde{\textbf{r}})$ as defined in equation (\ref{Arithmetical r-structure of F_n to F_{n+1}}) and equation (\ref{Arithmetical d-structure of F_n to F_{n+1}}) satisfies the recurrence relation for $F_{n+1}$. Now, note that by our choice of $(r_{k1} = r_0, r_{2n+1}, r_{2n+2}) \in \mathbb{N}^3$,  we see that 
$$
\begin{aligned}
r_{2n+1} &\mid r_0 + r_{2n+2}\\
r_{2n+2} &\mid r_0 + r_{2n+1}\\
r_0 &\mid \sum _{i=1}^{2n+2} r_i\\
\end{aligned}
$$
Since we already have the relation 
$$
\begin{aligned}
  r_0 &| r_1 + r_2 + \cdots + r_{2n}, \ for \ i=0 \\
  r_i &| r_0 + r_{i+1} \ for \ i\ odd\\
  r_i &| r_0 + r_{i-1} \ for \ i\ even,\\
\end{aligned}
$$ $\forall \ 0 < i \leq 2n$ for $F_n$. Hence we have $(\diag(\tilde{\textbf{d}}) - A(F_{n+1}))\tilde{\textbf{r}} = 0$.\\
Secondly, for the second part of Theorem \ref{subdivision and smoothing equation}  we need to observe that  we are removing $k^{th}$ arm $A_k$ keeping central vertex intact so that it will have no effect in $r$-labeling of central vertex i.e., $v_0$. Further, note that in removing arm $A_k$ it will have effect of removing the equation 
$$
\begin{aligned}
  d_kr_k &= r_0 + r_{k+1}\\
  d_{k+1}r_{k+1} &= r_0 + r_{k-1}\\
\end{aligned}
$$
from the recurrence relation for $F_n$. Now after this removal we will be left with recurrence relation with no $r_k$ and $r_{k+1}$ term but, it will not distort the divisibility property of other r-labeling and d-labeling in remaining vertices. Hence the remaining recurrence relation is exactly the recurrence relation of $F_{n-1}$. Hence the theorem.
\end{proof}
\\
\begin{corollary}
 Let $n \geq 3$ and  $(\textbf{d},\textbf{r}) \in Arith(F_n)$. Let $(r_{k1}, r_{2n+1}, r_{2n+2}) \in \mathcal{A}_{\textbf{r}}(C_3)$ be such that  $r_{k1} = r_0 $  where $\textbf{r} = (r_0, r_1, \cdots, r_{2n}) \in \mathcal{A}_{\textbf{r}}(F_n)$. Define integer vectors $\tilde{\textbf{d}}$ and $\tilde{\textbf{r}}$ of length $2n+3$ as follows:
\begin{equation}
\tilde{\textbf{r}_i} = \begin{cases}
  r_i \ for \ 0 \leq i \leq 2n,\\
   r_{2n+1} \ for \ i = 2n+1,\\
   r_{2n+2} \ for \ i = 2n+2,
\end{cases}
\end{equation}
\begin{equation} 
\tilde{\textbf{d}_i} = \begin{cases}
  d_i \ for \ 0 \leq i \leq 2n,\\
  \frac {r_0 + r_{2n+2}}{r_{2n+1}} \ for \ i = 2n+1,\\
   \frac {r_0 + r_{2n+1}}{r_{2n+2}} \ for \ i = 2n+2.
\end{cases} 
\end{equation}
Then $(\tilde{\textbf{d}},\tilde{\textbf{r}})$ is an arithmetical structure on $F_{n+1}$.
\end{corollary}

The next example construct some arithmetical structures on $F_3$ from arithmetical structures on $F_2$. 
\begin{exam}
Let $r=( 1,2,3,3,2)$ be an r-structure on $F_2$. Now, choose $r_k = (r_{k1},r_{k2},r_{k3}) \ \in \ Arith(C_3)$ such that $r_{k1} = 1$. Then for $r_k = (1,1,1)$ we have $\tilde{r} = (1,2,3,3,2,1,1)$, for $r_k = (1,1,2)$ we have $\tilde{r} = (1,2,3,3,2,1,2)$, for $r_k = (1,2,1)$ we have $\tilde{r} = (1,2,3,3,2,2,1)$.
\end{exam}

\subsection{Graph obtained by applying clique-star operation on fan graph}
First, we recall the concept of clique-star operation on a graph. The concept of clique-star operation  has been well studied in \cite{HCEV18}. Given a graph $G$ and a clique $C$(a set of pairwise adjacent vertices) of $G$, the clique-star transform of $G$, denoted by $cs(G, C)$, is the graph obtained from $G$ by deleting all the edges between the vertices in $C$ and adding a new vertex $v$ with all the edges between $v$ and the vertices in $C$. For example, if $G $ is the complete graph $K_n$ on $n$ vertices then the clique on $G$ is the set of all vertices of $K_n$. Then if we apply the clique star operation on $K_n$ we will get star graph $S_n$. Note that $cs(G,C)$ has one vertex more than the number of vertices in $G$.
\cite{HCEV18} has studied the relationship between the arithmetical structure on $G$ and $cs(G,C)$ and has proved that the critical groups of $G$ and $cs(G,C)$ are isomorphic. We refer the reader to \cite{HCEV18} and references therein for more details on clique-star transformation of graphs. There are some special cases of clique-star transformation which needs special mention. If in a graph $G$, clique $C$ has two vertices i.e an edge $e$. Then, this clique-star transformation is called subdivision. This is because of the fact that in case of 2-clique in the graph, if we apply the clique-star operation, it will have the effect of dividing that particular edge e(say) into two edges. Given a graph $G$, the subdivision of an edge $e = uv$ with endpoints $u$ and $v$ yields a graph containing one new vertex $w$, and with a pair of edges $uw$ and $vw$ replacing e. Secondly, if in graph $G$, clique $C$ has one vertex say $v$. Then, in this case if we apply the clique-star operation on $G$ then it have the effect of adding one new pendant edge to the graph $G$ at the position of vertex $v$.

Now, we will introduce some notations in clique-star operation on graph $G$. Let $cs_n(G,C)$ denote the graph obtained from $G$ by applying clique-star operation $n$ times in $G$. Note that, $cs_n(G,C)$ has $n$ vertices more than the number of vertices in $G$. We would like to recall some of the results from \cite{HCEV18} regarding the relationship between arithmetical structure on $G$ and $cs(G,C)$. Given $(\textbf{d},\textbf{r}) \in Arith (G)$, let
\begin{equation}\label{cliquearith}
\tilde{\textbf{d}}_u = cs(\textbf{d}, C)_u = \begin{cases}
  \textbf{d}_u \ &if \ u \notin C\\ 
  \textbf{d}_u + 1 \ &if \ u \in C\\
  1 \ &if \ u = v,
\end{cases}
\ \ \ 
\tilde{\textbf{r}}_u = cs(\textbf{r}, C)_u = \begin{cases}
\textbf{r}_u \ & if \ u \in V\\
\sum _{u \in C} \textbf{r}_u \ &if \ u = v,
\end{cases} 
\end{equation} where V is the set of vertices of $G$.

\begin{theorem} \cite{HCEV18} \label{Cique Star Transform by Corrales and Valencia} 
Let $G$ be a strongly connected graph, $C$ a clique of $G$, and $\tilde{G} = cs(G, C)$. If $\mathcal{A}'(\tilde{G})$ is the set of arithmetical structures $(\textbf{d}',\textbf{r}')$ of $\tilde{G}$ with $\textbf{d}'_v = 1$, then
$\mathcal{A}' (\tilde{G}) = \{ (\tilde{\textbf{d},\tilde{\textbf{r}}}) \ | \  ( \textbf{d},\textbf{r}) \ \in \mathcal{A}({G})\} $,  where $\mathcal{A}(G) = Arith(G)$. 
    \end{theorem}
    
It is further proved that $\Phi(G,\textbf{d},\textbf{r}) \cong \Phi(cs(G,C),\tilde{\textbf{d}},\tilde{\textbf{r}})$, see [ \cite{HCEV18}, Theorem 5.2, Theorem 5.3]. Again we recall following corollaries as the consequence of theorem \ref{Cique Star Transform by Corrales and Valencia} from \cite{HCEV18}.


\begin{corollary}\cite{HCEV18}
Given a graph $G$ and $v$ one of its vertices, let $\tilde{G}$ be the graph resulting from adding the edge $vv'$. If $(\textbf{d},\textbf{r})$ is an arithmetical structure of $G$, then
$$
\tilde{\textbf{d}}_u = cs(\textbf{d}, C)_u = \begin{cases}
  1 \ &if \ u = v'\\ 
  \textbf{d}_u + 1 \ &if \ u = v\\
  \textbf{d}_u \ &otherwise
\end{cases}
,\ \ \ \ 
\tilde{\textbf{r}}_u = cs(\textbf{r}, C)_u = \begin{cases}
\textbf{r}_v  \ & if \ u = v'\\
\textbf{r}_u \ &otherwise
\end{cases}
$$ is an arithmetical structure on $\tilde{G}$.
\end{corollary}

\begin{corollary}\cite{HCEV18}
Given a graph $G$ and  $e = u_1u_2$ one of its edges, let $\tilde{G}$ be the graph obtained by subdividing the edge $e$. If $(\textbf{d},\textbf{r})$ is an arithmetical structure of $G$, then
$$\tilde{\textbf{d}}_u = cs(\textbf{d}, C)_u = \begin{cases}
  1 \ &if \ u = v\\ 
  \textbf{d}_u + 1 \ &if \ u = u_1, u_2\\
  \textbf{d}_u \ &otherwise
\end{cases}
,\ \ \ \ 
\tilde{\textbf{r}}_u = cs(\textbf{r}, C)_u = \begin{cases}
\sum_{u \in e}\textbf{r}_u \ & if \ u = v\\
\textbf{r}_u \ &otherwise
\end{cases}
$$ is an arithmetical structure on $\tilde{G}$.
\end{corollary}


In view of the above theorems we have following observation. We would like to generalize the sequence of clique-star operation in graph $G$. Given $(\textbf{d},\textbf{r}) \in  Arith(cs_n(G)).$
$$
\tilde{\textbf{d}}_u = cs_{n+1}(\textbf{d}, C)_u = \begin{cases}
  \textbf{d}_u \ &if \ u \notin C\\ 
  \textbf{d}_u + 1 \ &if \ u \in C\\
  1 \ &if \ u = v
\end{cases}
,\ \ \ \ 
\tilde{\textbf{r}}_u = cs_{n+1}(\textbf{r}, C)_u = \begin{cases}
\textbf{r}_u \ & if \ u \in V\\
\sum _{u \in C} \textbf{r}_u \ &if \ u = v
\end{cases},
$$ where V is the set of vertices of $cs_{n}(G)$. Then we have the following theorem which relates the arithmetical structure on $cs_n(G)$ and $cs_{n+1}(G)$.

\begin{theorem} \label{Relation between cs_n and cs_{n+1}} 
Let $cs_n(G)$ be a strongly connected graph, $C$ a clique of $cs_n(G)$, and let $\tilde{G}_{n+1} =cs_{n+1}(G,C) =  cs(cs_n(G), C)$. If $\mathcal{A}'(\tilde{G}_{n+1})$ is the set of arithmetical structures $(\textbf{d}',\textbf{r}')$ of $\tilde{G}_{n+1}$ with $\textbf{d}'_v = 1$, then
$\mathcal{A}' (\tilde{G}_{n+1}) = \{ (\tilde{\textbf{d},\tilde{\textbf{r}}}) \ | \  ( \textbf{d},\textbf{r}) \ \in \mathcal{A}({cs_n(G)})\} $,  where $\mathcal{A}(cs_n(G)) = Arith(cs_n(G))$.
\end{theorem}

\begin{proof}
    Proof follows from repeated  use of theorem \ref{Cique Star Transform by Corrales and Valencia}.
\end{proof}

We see that the number of arithmetical structures on $cs_{n+1}(G,C)$ is at least the number of arithmetical structures on $cs_n(G,C)$, i.e., $$|Arith(cs_n(G,C))| < |Arith(cs_{n+1}(G,C))|.$$ Further, we note that $\Phi(cs_n(G,C),\textbf{d},\textbf{r}) \cong \Phi(cs_{n+1}(G,C),\tilde{\textbf{d}},\tilde{\textbf{r}})$ by \cite{HCEV18}, Theorem 5.2. We further observe that in a strongly connected graph $cs_n(G)$ obtained from strongly connected graph G after applying the clique-star operation n times. The set $Arith(cs_n(G))$ is in bijective correspondence with the set $\mathcal{A}' (\tilde{G}_{n+1}) = \{ (\tilde{\textbf{d},\tilde{\textbf{r}}}) \ | \  ( \textbf{d},\textbf{r}) \ \in \mathcal{A}({cs_n(G)})\} $, where $\mathcal{A}'(\tilde{G}_{n+1})$ is the set of arithmetical structures $(\textbf{d}',\textbf{r}')$ of $\tilde{G}_{n+1}$ with $\textbf{d}'_v = 1$.

We now turn to clique-star transformation of Fan graph. Fan graph $F_n$ on $2n+1$
vertices has $n$ number of choices for 3-clique, $3n$ number of choices for 2-clique and $2n+1$ number of choices for 1-clique. We would like to first take up the case for 3-clique. If we apply clique-star transformation on Fan graph $F_n$ with $2n+1$ vertices $n$ times with $C$ as 3-clique. We will get special type of graph having $2n + n + 1 = 3n + 1$ vertices. We will name this graph as arrow star graph and we denote it by $AS_{3n+1}$. Note that $AS_{3n+1}$ as no 3-cliques, of course it has 2-cliques and 1-cliques. We will further get new graphs if we go on  applying clique-star operation accordingly on $AS_{3n+1}$. See fig. \ref{Arrow Star Graph}.
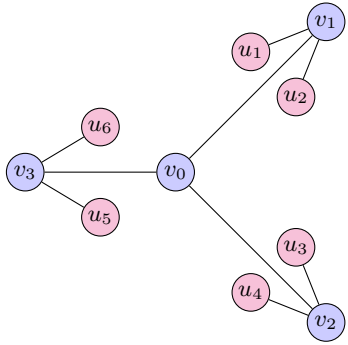
\begin{figure}[H] 
\centering
\begin{tikzpicture}[scale=2, every node/.style={circle, draw, fill=blue!20, minimum size=0.5cm, inner sep=0pt, font=\small}]

\node (A) at (0,0) {$v_0$};

\node (B) at (1,1) {$v_1$};
\node (C) at (1,-1) {$v_2$};
\node (D) at (-1,0) {$v_3$};
\node(E) [fill=magenta!30] at (0.5,0.8) {$u_1$};
\node (F) [fill=magenta!30] at (0.8,0.5){$u_2$};
\node(G) [fill=magenta!30] at (0.5,-0.8) {$u_4$};
\node (H) [fill=magenta!30] at (0.8,-0.5){$u_3$};
\node(I) [fill=magenta!30] at (-0.5,0.3) {$u_6$};
\node (J) [fill=magenta!30] at (-0.5,-0.3){$u_5$};
\draw (A) -- (B);
\draw (A) -- (C);
\draw (A) -- (D);
\draw (B) --(E);
\draw(B) --(F);
\draw (C) --(G);
\draw(C) --(H);
\draw (D) --(I);
\draw(D) --(J);
\end{tikzpicture}
\caption{ Arrow Star Graph $AS_{10}$ with 10 vertices obtained from Fan Graph $F_3$ by applying clique star transformation.}
\label{Arrow Star Graph}
\end{figure}

Now, we observe that there are at least $3n$ ways of applying clique-star transformation on $F_n$ with $C$ as 2-clique. Among all such choices we would like to apply clique-star transformation with $C$ as 2-clique on $F_n$ as follows. Since we know that $F_n$ has $n$ cyclic arms $A_i$ as $C_3$. We will apply subdivision of one edge on one of the arm say $A_i$ and then we will apply similar subdivision on next consecutive arm say $A_{i+1}$. We will repeat this process in anti-clockwise order until we reach the arm $A_i$. After applying $n$ number of sub-divisions on $F_n$ in this order we will get a graph similar like Fan graph having $n$ number of arms and $2n + 1 + n = 3n + 1$ vertices. But, this time the resulting graph will have cyclic arms $A_i$ as $C_4$.  We will name this new graph as $C_4$-Fan Graph or $C_4$-Fan Graph and denote it by $C_4F_n$. Note that $C_4F_n$ has no 3-clique but has 2-cliques and 1-cliques. See Fig. \ref{C_4F_3 Graph}.
\begin{figure}[H]
    \centering
\begin{tikzpicture}[scale=2, every node/.style={circle, draw, fill=blue!20,minimum size=0.5cm,, inner sep=0pt, font=\small}]

\node[fill=blue!30] (A) at (0,0) {$v_0$};

\node (B1) at (1,0) {$v_6$};
\node (B2) at (0.5,0.87) {$v_5$};
\node (B3) at (-0.5,0.87) {$v_4$};
\node (B4) at (-1,0) {$v_3$};
\node (B5) at (-0.5,-0.87) {$v_2$};
\node (B6) at (0.5,-0.87) {$v_1$};
\node [fill=magenta!30] (B7) at (0,1.8) {$u_1$};
\node [fill=magenta!30] (B8) at (-1.5,-0.9) {$u_2$};
\node [fill=magenta!30] (B9) at (1.5,-0.8) {$u_3$};
\draw[thick] (A) -- (B1);
\draw[thick] (A) -- (B2);
\draw[thick] (B6) -- (B9);

\draw[thick] (A) -- (B3);

\draw[thick] (A) -- (B4);

\draw[thick] (A) -- (B5);

\draw[thick] (A) -- (B6);
\draw[thick] (B1) -- (B9);

 \draw[thick] (B3) -- (B7);
 \draw[thick] (B2) -- (B7);
  \draw[thick] (B4) -- (B8);
 \draw[thick] (B5) -- (B8);
\end{tikzpicture}
    \caption{$C_4$ Fan graph $C_4F_3$ obtained from $F_3$ by applying Clique-star operation.}
    \label{C_4F_3 Graph}
\end{figure}
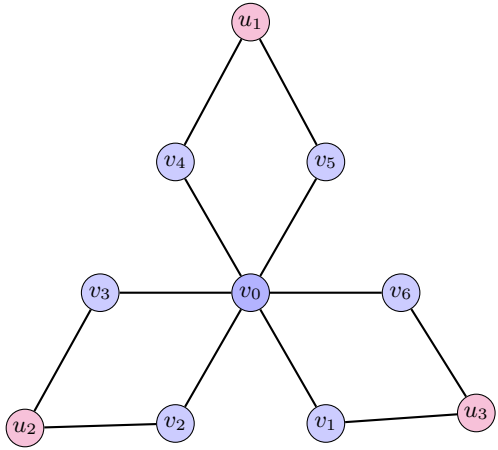

Finally, we observe that there are at least $2n+1$ ways of applying clique-star transformation on $F_n$ with $C$ as 1-clique. Among all such choices we would like to apply clique-star transformation on $F_n$ with $C$ as 1-clique as follows. Beginning with the central vertex of $F_n$ say $v_0$. We will add $2n+1$ number of  pendant vertices to $F_n$ beginning from central vertex $v_0$ in anti-clockwise direction. In this process we will get a well-known graph from $F_n$ called the corona graph with one pendant edge attached to each vertex of $F_n$. See Fig. \ref{Corona Graph on F_3}.

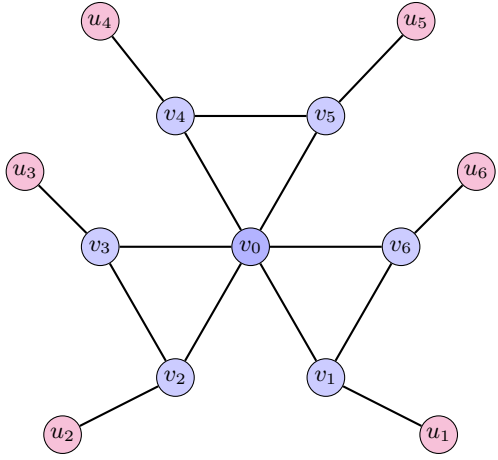
\begin{figure}[H]
\centering
\begin{tikzpicture}[scale=2, every node/.style={circle, draw, fill=blue!20, minimum size=0.5cm, inner sep=0pt, font=\small}]

\node[fill=blue!30] (A) at (0,0) {$v_0$};

\node (B1) at (1,0) {$v_6$};
\node (B2) at (0.5,0.87) {$v_5$};
\node (B3) at (-0.5,0.87) {$v_4$};
\node (B4) at (-1,0) {$v_3$};
\node (B5) at (-0.5,-0.87) {$v_2$};
\node (B6) at (0.5,-0.87) {$v_1$};
\node  [fill=magenta!30](B7) at (-1,1.5) {$u_4$};
\node  [fill=magenta!30] (B8) at (1.1,1.5) {$u_5$};
\node [fill=magenta!30] (B9) at (-1.5,0.5) {$u_3$};
\node [fill=magenta!30] (B10) at (1.5,0.5) {$u_6$};
\node [fill=magenta!30] (B11) at (-1.25,-1.25) {$u_{2}$};
\node [fill=magenta!30] (B12) at (1.25,-1.25) {$u_{1}$};
\draw[thick] (A) -- (B1);
\draw[thick] (A) -- (B2);

\draw[thick] (A) -- (B3);
\draw[thick] (B2) -- (B3);

\draw[thick] (A) -- (B4);

\draw[thick] (A) -- (B5);
\draw[thick] (B4) -- (B5);

\draw[thick] (A) -- (B6);

\draw[thick] (B1) -- (B6);
 \draw[thick] (B3) -- (B7);
 \draw[thick] (B2) -- (B8);
 \draw[thick] (B4) -- (B9);
 \draw[thick] (B1) -- (B10);
 \draw[thick] (B5) -- (B11);
 \draw[thick] (B6) -- (B12);
\end{tikzpicture}
\caption{Corona graph  on $F_3$ with one pendent vertex at each vertex of $F_3$.}
\label{Corona Graph on F_3}
\end{figure}

Now, we discuss one algorithm that produces arithmetical structures on any subdivision of a graph as given by \cite{HCEV18}. First, we recall few terminologies from \cite{HCEV18}. Let $V = V(G)$ and $e_v$ denote the characteristic vector of $v$ given by $(e_v)_u = \delta_{vu}$ for each $u \in V.$ Given $U \subset V$ with $m$ vertices, and order on $U$ is a bijective function $\theta : U \rightarrow [m],$ where $\theta$ can be described by the vector $(\theta^{-1}(1), \theta^{-1}(2), \cdots, \theta^{-1}(m)).$ Given $u, v \in V,$ let $P_{u,v}$ be the set of paths between $u$ and $v$. Also, given $\textbf{r} \in \mathbb{N}^V$ and $u \in V$, let $$ W_u(\textbf{r}) = \{w \in V \mid \ \exists \ P \in P_{u,w} \text{   such that } \textbf{r}_w \neq 0 \ and \ \textbf{r}_v = 0 
\text{ for every internal vertex } v \ of \ P \}.$$ Given a vector $\textbf{r} \in \mathbb{N}^V$, let $$ Supp(\textbf{r}) = \{v \in V \mid \textbf{r}_v \neq 0\}.$$
The algorithm will receive three inputs given by: Fan graph $F_n$, $\theta$ an order in  $V(C_4F_n)$  - $V(F_n),$ and 
$$(\textbf{r}_0)_u = \begin{cases}
    \textbf{r}_u &if u \in V(F_n)\\
    0 &otherwise
\end{cases}
$$

\begin{algorithm}[H]
\caption{Construction of $\textbf{r}(F_n, \textbf{r}_0, \theta) \in \mathbb{N}^{V({C_4F_n})}$  for  $C_4F_n$.}
\textbf{Input: A graph $F_n$, a vector $\textbf{r}_0 \in \mathbb{N}^{V(F_n)}$ and an order $\theta $ on $U = V(C_4F_n) - V(F_n)$.} \\
\textbf{Output: A vector $\textbf{r}(F_n, \textbf{r}_0, \theta) \in \mathbb{N}^{V({C_4F_n})}$}
\begin{algorithmic}
\item Set $ i = 1.$
\item While $ i \leq |U|$, set $u = \theta^{-1}(i),$ $\textbf{r}_i = \textbf{r}_{i-1} + \sum _{w \in W_{_{u(\textbf{r}_{i-1})}}} \textbf{r}_we_u $, and set $i = i+1.$
\item $Return :$ $\textbf{r}(F_n,\textbf{r}_0,\theta)$ = $\textbf{r}_{|U|}.$
 \end{algorithmic}
\end{algorithm}
 Now, we illustrate this algorithm with the help of an example.
    
\begin{exam}
Let $F_3$ be the fan graph with vertices labeled as $(v_0, v_1, v_2, v_3, v_4, v_5, v_6)$. Let $\textbf{r}_0 = ( 1,2,3,3,2,2,1)$. We subdivide the edges $v_1v_2$, $v_3v_4$, $v_5v_6$ once and obtain $C_4F_3$ , so that we have $\textbf{r}_0 = (v_0,v_1,0,v_2,v_3,0,v_4,v_5,0,v_6) =(1,2,0,3,3,0,2,2,0,1)$. Take $\theta$ an order on $(v_7,v_8,v_9)$. Here we have named added vertices as $v_7$, $v_8$ and $v_9$. Since $|U| = 3$, we will have three iterations. In iteration 1 we get $W_{\theta^{-1}(1)}(\textbf{r}_0) = W_{v_7}(1,2,0,3,3,0,2,2,0,1) =\{v_1,v_2\}$ and $\textbf{r}_1 = (1,2,5,3,3,0,2,2,0,1) $. In iteration 2 we get $W_{\theta^{-1}(2)}(\textbf{r}_1) = W_{v_8} (1,2,5,3,3,0,2,2,0,1) = \{v_3,v_4\}$ and $\textbf{r}_2 = (1,2,5,3,3,5,2,2,0,1)$. In iteration 3 we get $W_{\theta^{-1}(3)}(\textbf{r}_2) = W_{v_9} (1,2,5,3,3,5,2,2,0,1) = \{v_5,v_6\}$ and $\textbf{r}_3 = (1,2,5,3,3,5,2,2,3,1)$. Thus, $\textbf{r} (F_3,\textbf{r}_0,\theta) = (1,2,5,3,3,5,2,2,3,1).$
\end{exam}

Moreover, if we subdivide the edges of $C_4F_3$ further, we will get more arithmetical structures. For instance we can start with $\textbf{r}_0 =( 1,1,1,1,1,1,1,1,1,0,1)$, $\textbf{d} = ( 6,2,2,2,2,2,2,2,2,2)$. Then we get $\textbf{r} = (1,1,1,1,1,1,1,1,1,2,1)$ and $\textbf{d} = (6,2,2,2,2,2,2,2,3,1,3)$. 

Motivated by the above algorithm, we write an algorithm to obtain arithmetical structures on any graph $cs(G)$ obtained by applying clique-star transformation on some graph $G$. For example, we give two such algorithms below to obtain some arithmetical structures on the arrow star graph $AS_{3n+1}$ and the corona graph on $F_n$ with one pendent vertex.

The algorithm will receive four inputs given by: Fan graph $F_n$, $\theta$ an order in $V(AS_{3n+1}) - V(F_n)$ or $V(corona \ graph \ on \ F
_n \ with \  one \ pendent \ vertex) - V(F_n),$ clique $C \subseteq V(F_n)$ and 
$$(\textbf{r}_0)_u = \begin{cases}
    \textbf{r}_u &if u \in V(F_n)\\
    0 &otherwise
\end{cases}
$$
\begin{algorithm}[H]
\caption{Algorithm for construction of $\textbf{r}(F_n,\textbf{r}_{i-1}, \theta)$ of  $ AS_{3n+1}.$}
 \textbf{Input :  A graph $F_n$, a vector $\textbf{r}_0 \in \mathbb{N}^{V(F_n)}$ and an order $\theta $ on $U = V(AS_{3n+1}) - V(F_n)$ and $n$-number of 3-clique $C$ in $F_n$.}\\
\textbf{Output :  A vector $\textbf{r}(F_n, \textbf{r}_0, \theta, C) \in \mathbb{N}^{V({AS_{3n+1}})}$.}

\center \underline {\textbf{Algorithm}}

\begin{algorithmic}
\item Step 1: Take $n$ number of 3-clique in $F_n$ numbered as $C_i$.

\item Step 2: Replace all the cliques in step 1 by star graph with centers $u_1$, $\cdots$, $u_n$.

\item Step 3: Set $\textbf{r}_0 = ( v_0, \cdots, v_{2n}, u_1, \cdots, u_n)$.

\item Step 4: Set $i = 1$.

\item Step 5: While $ i \leq |U|$, set $(\textbf{r}_i)_v = \begin{cases}
   \sum _{v \in C_i} \textbf{r}_v & if v = u_j\\
\textbf{r}_v & otherwise
\end{cases}$ and set $i = i+1$\\
   Step 6 : Return: $\textbf{r}(F_n,\textbf{r}_{i-1}, \theta)$.
   \end{algorithmic}
   \end{algorithm}
   We illustrate this with the help of one example.

\begin{exam}  
Consider $F_2$. Now we can see that $F_2$ has $2$ number of 3-cliques. Take 3-clique and replace the clique by star graph with centers $u_1$ and $u_2$. $|U| = 2$ so we will have 2 iterations. Take $\textbf{r}_0 = (1,2,3,2,3)$. Now we will have $\textbf{r}_1 = (1,2,3,2,3,6)$ and $\textbf{r}_2 = (1,2,3,2,3,6,6) = \textbf{r}( F_2,\textbf{r}_0,\theta,C)$. Proceeding in this way we can produce some arithmetical structures on $F_n$.
\end{exam}

The Algorithm for the corona graph on $F_n$ with one pendent vertex on vertices of its arms is similar to algorithm for $AS_{3n+1}$ except that clique here will be 1-clique, i.e.,  $C = \{v_i, i \neq 0\}.$ In general we cannot obtain all the arithmetical structures on $AS_{3n+1}$, $C_4F_n$ and corona graph on $F_n$ with one pendent vertices by the above algorithm, but it will certainly give lower bound on number of arithmetical structures on this graphs. In \cite{HCEV18}, Proposition 6.8, if graph $s(G)$ is obtained from graph $G$ by subdividing its edges $e \in E(G)$ $n_e$ times then $$ |\mathcal{A}(s(G))| \geq (|\mathcal{A}(G)| - 1) \ . \ \Pi_{e \in E(G)} C_{n_e} \ + \ \Pi_{e \in E(G)} C_{n_e} + 1, $$ where $C_n$ is the catalan number. In particular, $$ |\mathcal{A}(C_4F_n)| \geq (|\mathcal{A}(F_n)| - 1) \ . \ \Pi_{e \in E(F_n)} C_{n_e} \ + \ \Pi_{e \in E(F_n)} C_{n_e} + 1.$$ Note that to get $AS_{3n+1}$ from $F_n$ we are subdividing one edge of cyclic arm once. Hence we have $n_e = 1$ for  one $e$ in $A_i$ and $n_e = 0$ for other two edges. Therefore, 
 $$ |\mathcal{A}(C_4F_n)| \geq (|\mathcal{A}(F_n)| - 1) +2.$$ Roughly, we can say $$ |\mathcal{A}(AS_{3n+1})| \geq (|\mathcal{A}(F_n)| - 1) $$ and $$ |\mathcal{A}(Corona \ graph \ on \ F_n)| \geq (|\mathcal{A}(F_n)| - 1) .$$

\section{Arithmetical Structures on Star Graphs}
Let $S_n$ denote the star graph with $n$ vertices $v_1, v_2, \cdots, v_n$ connected to one central
vertex labeled $v_{0}.$  It is mentioned by 
Corrales and Valencia [7], arithmetical structures on $S_n$ are in bijection with solutions in the positive integers of the  Diophantine equation
$$d_0 = \sum_{i=1}^{n} \frac{1}{d_i}. $$

Consider the star graph $S_n$. Then the adjacency matrix of $S_n$ is given by 
$$A_{n+1} =
    a_{v_i,v_j}=\begin{cases}
        1, \,\, \mbox{if} \,\, i=0, j=1,\ldots,n, \\
         1, \,\, \mbox{if} \,\, j=0, i=1,\ldots,n, \\
        0, \,\, \mbox{otherwise}.
    \end{cases}. $$
Note that \textbf{(d,r)} is an arithmetical structure on $S_n$ iff $(D-A_{n+1})\textbf{r} = 0$, where $D = \mathrm{diag} (d_0, d_1, d_2, \cdots, d_{n-1}, d_{n})$ is a diagonal matrix iff the following equalities hold:
\begin{equation}\label{srdstructure}
\begin{aligned}
    d_i r_i & = r_0  \\
         d_0 r_0 &  = r_1+ r_2+ r_{3}+ \cdots +r_n
\end{aligned} 
\end{equation}
Thus we have a primitive positive integer $\textbf{r}= (r_0, r_1, r_2, \ldots, r_n) $  is an arithmetical $r$-structure on $S_n$ iff
\begin{equation}\label{rstructure1}
\begin{aligned}
 & r_i |r_0 \text{  for all   } i \in [n]\\
 & r_0 | r_1+r_2+r_3 +\cdots+r_n   
\end{aligned}
\end{equation}

Observe that $\sum_{i=1}^{n} \frac{1}{d_i} = \sum_{i=1}^{n} \frac{r_i}{d_ir_i}
= \sum_{i=1}^{n} \frac{r_i}{r_0} = d_0$, which is mentioned in Corrales and Valencia [7]. Also,  in the reverse direction, given a solution of the form $d_0 = \sum_{i=1}^{n} \frac{1}{d_i} $, setting $r_0 = lcm(d_1, d_2, \ldots, d_n)$ and $r_i =r_0/d_i$ for all $i \in [n]$ gives an arithmetical structure on $S_n$. Solutions of Diophantine equation have been much
studied, sometimes under the name Egyptian fractions with the assumption that the $d_i$'s are distinct, but many open questions about them remain, see, for example \cite{Egyptianfraction1,EgyptianfractionNew2,Egyptianfraction4}, and the references therein. 
Note that the numbers of solutions for $d_0 = \sum_{i=1}^{n} \frac{1}{d_i}$  for $n \leq 8$ are given by sequence A280517 in \cite{Sloane18}. The number of solutions for $n \leq 8$ is as follows: $1, 2, 14, 263, 13462, 2104021.$

\begin{theorem}\label{arithSn-1to Sn+1}
 Let $S_{n-1}$ be the star graph with $n$ vertices $v_1, \ldots, v_n$, where $v_1$ is centre vertex.  If $(\d,\r)$ is an arithmetical structure on $S_{n-1}$, then $S_n$ has an arithmetical structure $(\tilde{\d},\tilde{\r})$ given by 
 $$  \tilde{\textbf{d}_i}  = \begin{cases}
         1, & \text{if}\;i = n+1\\
         d_{1}+ 1, & \text{if}\; i = 1\\
        \textbf{d}_{i},  & \text{Otherwise.}\\    
   \end{cases}$$
   $$\tilde{\textbf{r}_i}  =  \begin{cases}
       r_{1}, &\text{if}\;  i = n+1 \\
        \textbf{r}_{i}, & Otherwise.
   \end{cases}
$$
\end{theorem}

\begin{proof}
    Let $A$ be the adjacency matrix of $S_{n-1}$. Observe that adjacency matrix of $S_n$ is given by 
   $$\tilde{A}= \begin{pmatrix}
    A & e_1 \\
    e_{1}^{T} & 0
\end{pmatrix},$$ where $e_1$ is the column vector with $1$-th entry is $1$ and all other entries are $0$.
 By direct computation one can verify that $(\diag(\Tilde{\d}) - \Tilde{A} )\tilde{\r} = 0$. Hence, the result follows. 
\end{proof}

Now, consider a complete graph $K_n$ on $n$ vertices then the clique on $G$ is the set of all vertices of $K_n$. Then by using the clique-star operation on $K_n$ we get star graph $S_n$. Thus we have the following result.

\begin{theorem}\label{arithKntoSn}
Let $K_n$ be a complete graph with vertex set $V= \{v_1, v_2, \ldots, v_n\}$ and $(\textbf{d},\textbf{r}) \in Arith (K_n)$. Define
$$\tilde{\textbf{d}}_u  = \begin{cases}
  \textbf{d}_u + 1 \ &if \ u \in V\\
  1 \ &if \ u = v,
\end{cases}
\ \ \ \ 
\tilde{\textbf{r}}_u  = \begin{cases}
\textbf{r}_u \ & if \ u \in V\\
\sum _{u \in V} \textbf{r}_u \ &if \ u = v,
\end{cases}
$$ where V is the set of vertices of $G$.   Then $(\tilde{\textbf{d}},\tilde{\textbf{r}} )$ is an arithmetical structure on $S_n.$
\end{theorem}

\begin{proof}
Consider the clique as the set of all vertices of  $K_n$. Then the result directly follows from equation~\ref{cliquearith}. 
\end{proof}

\begin{exam}
Consider $K_3$, which corresponds to the cycle $C_3$. From \cite{BHDNJC18}, the arithmetical structure $(\d, \r)$ on $C_3$ is given by 
\begin{center}
    \begin{tabular}{ | l | l | l| p{5cm} | }
    \hline
   $\d = (d_1, d_2, d_3)$ & \r= $(r_1, r_2, r_3)$  & \#\\
   \hline
(2, 2, 2) & (1, 1, 1) & 1 \\
(3, 3, 1) & (1, 1, 2) & 3 \\
(5, 2, 1) & (1, 2, 3) & 6 \\
     \hline
    \end{tabular}
\end{center}
Now, by Theorem~\ref{arithKntoSn} some of the arithmetical structures $(\tilde{\d}, \tilde{\r})$ of $S_3$ are the following:
\begin{center}
    \begin{tabular}{ | l | l | l| p{5cm} | }
    \hline
   $\tilde{\d} = (\tilde{d}_0, \tilde{d}_1, \tilde{d}_2, \tilde{d}_3)$ & $\tilde{\r}= \tilde{r}_0, \tilde{r}_1, \tilde{r}_2,\tilde{r}_3)$  & \#\\
   \hline
(1, 3, 3, 3) & (3, 1, 1, 1) & 1 \\
(1, 4, 4, 2) & (4, 1, 1, 2) & 3 \\
(1, 6, 3, 2) & (6, 1, 2, 3) & 6 \\
\hline
(3, 1, 1, 1) & (1, 1, 1, 1) & 1 \\
     \hline
    \end{tabular}
\end{center}
Note that $S_3$ has $14$ arithmetical structures.
\end{exam}


Using \ref{srdstructure} and Theorem~ (\ref{arithSn-1to Sn+1}) we have the following. 

\begin{exam}\label{ASS4}
The star graph $S_4$ has $263 $ arithmetical structures. We have listed $232$, namely 
\begin{center}
    \begin{tabular}{ | l| l| p{5cm} | }
    \hline
$\tilde{\d} =(d_0, d_1, d_2, d_3, d_4)$ & $\tilde{\r} = (r_0, r_1, r_2, r_3, r_4)$  & \# \\
   \hline
  (1, 4, 4, 4, 4) & (4, 1, 1, 1, 1) &  1 \\
(1, 6, 3, 6, 3) & (6, 1, 2, 1, 2) &  6 \\
 (1, 6, 6, 6, 2) & (6, 1, 1, 1, 3) & 4\\
 (1, 12, 3, 3, 4) & (12, 1, 4, 4, 3) & 12 \\
 (1, 12, 2, 6, 4) & (12,1, 6, 2, 3) &  24 \\
 (1,8, 8, 4, 2) & (8,1, 1, 2, 4) & 12\\
 (1,12, 12, 2, 3) & (12,1, 1, 6, 4) & 12 \\
  (1,10, 5, 5, 2, 1) & (10,1 2, 2, 5) &  12\\
 (1, 12, 4, 6, 2 )  & ( 12, 1 , 3 , 2, 6) & 24\\ 
(1, 6, 3, 4, 4) & (12 , 2, 4, 3, 3) & 12\\
( 1, 18, 9, 3, 2) & ( 18, 1, 2, 6, 9) & 24\\
 ( 1, 20, 4, 5, 2) & ( 20, 1, 5, 4, 10) & 24\\
( 1, 24, 8, 3, 2) & ( 24, 1, 3, 8, 12) & 24\\
 ( 4, 1, 1, 1, 1 ) & (1, 1, 1, 1, 1) &   1\\
( 2, 3, 3, 3, 1 ) & (3, 1, 1, 1, 3) & 4  \\
 ( 2, 4, 4, 2, 1 ) & (4, 1, 1, 2, 4) &  12 \\
 ( 2, 6, 3, 2, 1 ) & (6, 1, 2, 3, 6) &  24  \\
 \hline
    \end{tabular}
\end{center}
\end{exam}

\subsection{Relation between Arrow Star graph $AS_n$ and star graph $S_n$}

The star graph $S_n$ consists of $n$-leaves $v_1$, $\cdots$, $v_n$ attached to a central vertex  $v_0$, and it can be obtained from the complete graph $K_n$ on n-vertices by a clique-star transformation. Similarly, the arrow-star graph $AS_n$ is derived from the fan graph $F_n$ through a series of suitable clique-star transformations.

To derive $AS_n$ from $S_n$, we apply a specific transformation: attach two pendant edges $u_i$ and $u_j$ to each leaf  $v_i$ of $S_n$. Repeating this process $2n$ times results in $AS_n$.

Alternatively, $AS_n$  can also be derived directly from $F_n$  by performing $n$ clique-star transformations using a $3$-clique $C$. Thus, 
$AS_n$ can be constructed either by $n$ clique-star transformations from 
$F_n$ or by $2n$ pendant-edge additions to $S_n$.

Therefore, $AS_n = cs_n(F_n,C)$ and $AS_n = cs_{2n}(S_n,C),$ which implies $cs_n(F_n,C) = cs_{2n}(S_n,C).$

So, we have the following observation:

\begin{observation}
We have
\begin{enumerate}
 \item $\mathcal{A}(F_n) \subset \mathcal{A}(AS_n)$.  
 \item $\mathcal{A}(S_n) \subset \mathcal{A}(AS_n)$.
 \item $\mathcal{A}(cs_n(F_n,C)) = \mathcal{A}(cs_{2n}(S_n,C)) $. 
 \item $|\mathcal{A}(F_n)| < |\mathcal{A}(AS_n)|$ and  $|\mathcal{A}(S_n)| < |\mathcal{A}(AS_n)|$.
\end{enumerate}
\end{observation}



Recall that the arithmetical structures on $S_n$ correspond bijectively to the positive integer solutions of the equation 
$ \sum _{i=1} ^ n d_i = d_0,$ where $(d_0, d_1, \cdots, d_n)$ is a 
$d$-labeling of the star graph $S_n$ with central vertex $v_0.$

Building on this, a portion of the set $\mathcal{A}(AS_n)$ (arithmetical structures on the arrow-star graph $(AS_n)$ ) is in bijection with Egyptian fraction representations. This connection arises due to the structural similarities between the labeling rules and the decomposition of a number into a sum of unit fractions, as demonstrated by Theorem~\ref{Relation between cs_n and cs_{n+1}}. Thus we have the following theorem.

\begin{theorem}
Given a star graph $S_n$ and $v$ one of its pendant  vertex, let ${cs(S_n)}$ be the graph resulting from adding the edge $vv'$. If $(\textbf{d},\textbf{r})$ is an arithmetical structure of $S_n$, then
$$
\tilde{\textbf{d}}_u = cs(\textbf{d}, C)_u = \begin{cases}
  1 \ &if \ u = v'\\ 
  \textbf{d}_u + 1 \ &if \ u = v\\
  \textbf{d}_u \ &otherwise
\end{cases}
,\ \ \ \ 
\tilde{\textbf{r}}_u = cs(\textbf{r}, C)_u = \begin{cases}
\textbf{r}_v  \ & if \ u = v'\\
\textbf{r}_u \ &otherwise
\end{cases}
$$ is an arithmetical structure on $cs({S_n})$. In particular if $\mathcal{A}'(cs(S_n))$ is the set of arithmetical structures $(\textbf{d}',\textbf{r}')$ of $cs(S_n)$ with $\textbf{d}'_v = 1$, then
$\mathcal{A}' (cs(S_n)) = \{ (\tilde{\textbf{d},\tilde{\textbf{r}}}) \ | \  ( \textbf{d},\textbf{r}) \ \in \mathcal{A}(S_n)\} .$
\end{theorem}

\section{Critical Group of Arithmetical Structure on Fan Graphs}
Let $L$ be the Laplacian matrix of a graph $G$. Recall that we can interpret $L$  as a linear map from $ \mathbb{Z}^n \rightarrow \mathbb{Z}^n, $ with its kernel generated by the transpose of the all-ones vector $\textbf{r} = (1, 1, \cdots, 1)$. The cokernel  $\Z^n/\Im(L)$ is a finitely generated abelian group. The critical group $ \Phi(G)$ of a given arithmetical structure $(\textbf{d}, \textbf{r})$ on  $G$ is defined as the torsion part of the cokernel of $L$. To compute this group, key properties of the Smith normal form of a matrix.

Given an $ n\times n$ matrix $M$ with integer entries, there exist invertible $n\times n$  integer matrices $S$ and $T$ such that $SMT = D,$ where $D = \diag(\alpha_1, \alpha_2, \cdots, \alpha_t, 0, 0, \cdots, 0)$ is a diagonal matrix, $t = rank (M)$, each $\alpha_k$ is a positive integer, and $\alpha_k$ divides $\alpha_{k+1}$ for all $k \in [t-1]$. This diagonal matrix is the Smith normal form of $M$, and the entries $\alpha_k$ are the invariant
factors of $M$. The cokernel $\Z^n/\Im(M)$ is isomorphic to $\mathbb{Z}^{n-t} \oplus \  (\bigoplus _{k=1}^t \mathbb{Z} / \alpha_k\mathbb{Z})$. Writing $D_k = D_k(M)$  as the greatest common divisor of all $k \times k$ minors of $M$ and defining $D_0(M) = 1$. We have $\alpha_{k} = D_{k}(M)/ D_{k-1}(M)$. In the case of $L(G,\textbf{d})$, the matrix $L = \diag(\textbf{d}) - A$ associated to an arithmetical structure $(\textbf{d}, \textbf{r})$ on a graph $G$ has rank $n-1$ as showed by Lorenzini [\cite{LD89}, Proposition 1.1]. We get that in this case the cokernel $\Z^n/\Im(L)$ is isomorphic to $\mathbb{Z} \oplus \  (\bigoplus _{k=1}^{n-1} \mathbb{Z} / \alpha_k\mathbb{Z})$. The critical group of $\Phi (G)$ is defined to be this torsion part of the cokernel $\Z^n/\Im(L)$. i.e.,
$$ \Phi (G) \cong \bigoplus _{k=1}^{n-1} \mathbb{Z} / \alpha_k\mathbb{Z}.$$ The order of $\Phi(G)$ is $D_{n-1}(M) = \alpha_1 \cdots \alpha_{n-1} $, and this integer is a well-known graph invariant, the number $\kappa(G)$, the number of spanning trees of $G$, also called the complexity of $G$. We have factorization of $\kappa(G) = D_1 . D_2/D_1...D_{n-1}/D_{n-2}$.

The group $\phi(G)$ can  be generated by at most $n - 1$ generators and  at most $\beta(G) = m - n + 1$ generators \cite{lorenzini1991finite}, page 281. The relationship between the eigenvalues of $L$ and the structure of the critical group $\Phi(G)$ has also been studied (see \cite{lorenzini2008criticalgroupandeigenvalues}). The eigenvalues of $L$ do contribute to the structure of $\Phi(G)$ but do not completely determine it. For certain graph families, the critical group is known to have specific properties: it is trivial for path graphs \cite{BHDNJC18}, cyclic for cycle graphs \cite{BHDNJC18}, and its structure for complete graphs $K_n$ and star graphs $S_n$ is discussed in \cite{criticalgroupon(starandcompletegraph)}. The critical group of arithmetical structures on conference graphs with $n$ squarefree vertices is studied in \cite{lorenzini2008criticalgroupandeigenvalues}. Determining the explicit structure of $\Phi(G)$ for a given family of graphs is not always straightforward, and recent literature has focused on explicitly determining the structure of these groups across various graph classes.

Now, the next result shows that critical group of $S_{n-1}$ and $S_n$ are isomorphic. That is,  the structure of the critical group remains unchanged when transitioning from a star graph with $n-1$ vertices to one with $n$ vertices, showcasing a certain invariance under this graph transformation.

\begin{theorem}
Let $S_{n-1}$ be the star graph with $n$ vertices $\{v_1, \ldots, v_n\}$, where $v_1$ is centre vertex. If $(\textbf{d}, \textbf{r})$ is an arithmetical structure on $S_{n-1}$, and  $(\tilde{\textbf{d}}, \tilde{\textbf{r}})$ is an arithmetical
structure of $S_{n}$,  then $\Phi(S_{n-1})\cong \Phi(S_{n}). $
\end{theorem}

\begin{proof}
Let $A$ be the adjacency matrix of $S_{n-1}$ and $\tilde{A}$  be the adjacency of $S_n$. The we have $$\diag(\tilde{\textbf{d}} )- \tilde{A} = \begin{pmatrix}
 d_{1}+1 &  & & & & \\
&   d_{2}  & & & & \\
&    & & \ddots &  & \\
&    & &  & d_{n} & \\ 
&    & &  &  & 1\\ 
\end{pmatrix} - \begin{pmatrix}
    A & e_1\\
    e_{1}^{T} & 1
\end{pmatrix} $$

$$ = \begin{pmatrix}
    I_n & e_1\\
    0 & 1
\end{pmatrix} \begin{pmatrix}
    E_{11} + \diag(\textbf{d}) - A & -e_1\\
    -e_{1}^{T} & 1
\end{pmatrix} \begin{pmatrix}
    I_n & 0\\
    e_{1}^{T} & 1
\end{pmatrix} $$ $$= \begin{pmatrix}
    I_n & e_1\\
    0 & 1
\end{pmatrix} \begin{pmatrix}
    E_{11} + \diag(\textbf{d}) - A -E_{11} & -e_1\\
    0 & 1
\end{pmatrix} 
 = \begin{pmatrix}
    diag(\textbf{d}) - A & 0\\
    0 & 1
\end{pmatrix},$$ where $E_{11}$ is $n \times n$ matrix with $(1,1)$th entry is $1$ and all other entries are zero.
So, $diag(\tilde{\textbf{d}} )- \tilde{A}$ and $\begin{pmatrix}
    \diag(\textbf{d}) - A & 0\\
    0 & 1
\end{pmatrix}$  are integrally equivalent. Hence proved.
\end{proof}

Now, we will  find structure of the critical group of arithmetical structures on the fan graph $F_n$.
Note that Fan graph $F_n$ is made up of cycle graph $C_3$ by gluing exactly one vertex of the cycle graph with each  other. To determine the structure of $\Phi(G)$ in case of $F_n$ we will observe the following facts:

Let $G$ and $G'$ be any two connected graphs. Let $v$ denote a vertex of $G$ and $w$ denote a vertex of $G'$. Let $\Gamma$ be the graph obtained as the union of $G$ and $G'$ with the vertex $v$ identified with the vertex $w$. Number the vertices of $G $ as $v_1$, $v_2$, $\cdots$, $v_i = w$ and vertices of $G'$ as $v_i = w$, $v_{i+1}$, $\cdots$, $v_{n}$. We see that Laplacian $L$ of $\Gamma$ is almost made up of two blocks. By adding all the rows of $L$ to the $i^{th}$ row and all the columns of $L$ to the $i^{th}$ column, $L$ becomes equivalent to a matrix made up of two disjoint blocks. The Smith normal form of the top left (respectively bottom right) corner block can be used to compute the Smith normal form of the laplacian of $G$ ( respectively $G')$. Hence we can conclude that $\Phi(\Gamma) \cong \Phi(G) \times \Phi(G')$. 

Hence, we have the following theorem.

\begin{theorem}
 Let $F_n$ be a fan graph on $2n+1$ vertices $v_0$, $v_1$, $\cdots$, $v_{2n}$ where $v_0$ is the central vertex with degree $2n$. Let $A_i$ be the $i^{th}$ cyclic arm of $F_n$ with vertices $v_{i1}$, $v_{i2}$ $v_{i3}$. Let vertex $v_{i1}$ of arm $A_i$ be glued with vertex $v_{(i+1)1}$ of arm $A_{i+1}$ for $ 0 < i < n$. Then $\Phi(F_n) \cong \Phi (A_1) \times \Phi (A_2) \times \cdots \times \Phi(A_n)$.
\end{theorem}

\begin{proof}
The Proof follows from the previous observation.   
\end{proof}

Here we would like to recall the result from \cite{BHDNJC18} regarding the critical group of a cycle graph. For an arithmetical $r$-structure $\textbf{r}$, let $\textbf{r}(1) = \# \{i: r_i=1\}. $
We then refer to Theorem 29 from \cite{BHDNJC18}, which states:

Let $(\textbf{d},\textbf{r}) \in Arith(C_n)$ be an arithmetical structure on the cycle. Then $$\textbf{r}(1) = 3n - \sum _{j=1} ^n d_j$$ and $\Phi(C_n, \textbf{d}, \textbf{r}) = \mathbb{Z}_{\textbf{r}(1)} $. 

Now, for $n = 3$, we have a specific case of this theorem, which is stated as follows:

For $n = 3$, let $(\textbf{d},\textbf{r}) \in Arith(C_3)$ be an arithmetical structure on the cycle. Then $$\textbf{r}(1) = 9 - \sum _{j=1} ^3 d_j$$ and $\Phi(C_3, \textbf{d}, \textbf{r}) = \mathbb{Z}_{\textbf{r}(1)}$. 

We know that $|Arith(C_3)| = 10.$ Thus, we have the following table:
\begin{table}[H]
    \centering
\begin{tabular}{|c|c|c|}
  $\d$  & $\r$  & $\Phi(C_3, \textbf{d}, \textbf{r})$ \\
     \hline
  (2,2,2)  & (1,1,1)&  $\mathbb{Z}_{3}$ \\
  (5,2,1)   & (1,2,3) & $\{e\}$ \\
  (2,1,5) & (2,3,1) & $\{e\}$ \\
  1,5,2) & (3,1,2) & $\{e\}$ \\
  (5,1,2)& (1,3,2) & $\{e\}$ \\
  (1,2,5)& (3,2,1) & $\{e\}$ \\
  (2,5,1) & (2,1,3) & $\{e\}$ \\
  (3,3,1) & (1,1,2) & $\mathbb{Z}_2$ \\
  (3,1,3) & (1,2,1) & $\mathbb{Z}_2$\\
  (1,3,3)& (2,1,1) & $\mathbb{Z}_2$\\
\end{tabular}
    \caption{Critical group of $C_3$}
    \label{cgfan}
\end{table}
With this observation, we can write down the critical group of $F_n$. i.e., $\Phi(F_n) =  \Phi (A_1) \times \Phi (A_2) \times \cdots \times \Phi(A_n), $ where $\Phi(A_i) \in \{\mathbb{Z}_3, \mathbb{Z}_2, \{e\} \}, i =1, 2, \ldots, n.$

In particular, we can determine the possible critical groups of the fan graph 
$F_2$, whose $r$-structure is derived from the $r$-structure of the cycle graph $C_3.$ The critical group of  $F_2$ is given by:
$$\mathbb{Z}_3\times \mathbb{Z}_2, \,\, \mathbb{Z}_3\times \{e\}, \,\,\mathbb{Z}_2 \times \{e\}, \,\,\mathbb{Z}_3\times \mathbb{Z}_3, \,\,\mathbb{Z}_2 \times \mathbb{Z}_2, \,\,\{e\}.$$


\section{Conclusion}
In conclusion, this study provides a detailed analysis of arithmetical structures on Fan Graphs $F_n$, shedding light on their combinatorial properties and the underlying algebraic framework. Further, the investigation of the critical group associated with these arithmetical structures enhances the comprehension of their algebraic and combinatorial significance. 

\vspace{0.5cm}
\noindent{\bf Acknowledgment.}
The authors gratefully acknowledge the financial support provided by Sikkim University, Sikkim. 

\vspace{0.5cm}
\noindent{\bf Compliance with Ethical Standards.}
\begin{itemize}
\item {\bf Data Availability Statement:} Our manuscript has no associated data. 

\item {\bf Conflict of Interests:} The authors declare no conflict of interest.\\
 
\end{itemize}

\bibliographystyle{plain}
\bibliography{main}
\end{document}